\documentclass[12pt,reqno]{amsart}

\usepackage{times}
\usepackage[T1]{fontenc}
\usepackage{mathrsfs}
\usepackage{latexsym}
\usepackage[dvips]{graphics}
\usepackage[titletoc,title]{appendix}
\usepackage{epsfig}

\usepackage{amsmath,amsfonts,amsthm,amssymb,amscd}
\input amssym.def
\input amssym.tex
\usepackage{color}
\usepackage{hyperref}

\usepackage{color}
\usepackage{breakurl}
\usepackage{comment}
\newcommand{\bburl}[1]{\textcolor{blue}{\url{#1}}}

\newcommand{\be}{\begin{equation}}
\newcommand{\ee}{\end{equation}}
\newcommand{\bea}{\begin{eqnarray}}
\newcommand{\eea}{\end{eqnarray}}

\newtheorem{thm}{Theorem}[section]

\newtheorem{lem}[thm]{Lemma}

\numberwithin{equation}{section}

\providecommand{\floor}[1]{\lfloor #1 \rfloor}

\begin{document}

\title{When the Nontrivial, Small Divisors of a Natural Number are in Arithmetic Progression}

\author{H\`ung Vi\d{\^e}t Chu}
\email{\textcolor{blue}{\href{mailto:hungchu2@illinois.edu}{hungchu2@illinois.edu}}}
\address{Department of Mathematics, University of Illinois at Urbana-Champaign, Urbana, IL 61820, USA}


\begin{abstract}
Iannucci considered the positive divisors of a natural number $n$ that do not exceed $\sqrt{n}$ and found all forms of numbers whose such divisors are in arithmetic progression. In this paper, we generalize Iannucci's result by excluding the trivial divisors $1$ and $\sqrt{n}$ (when $n$ is a square). Surprisingly, the length of our arithmetic progression cannot exceed $5$. 
\end{abstract}

\maketitle
\section{Introduction and main results}
For each $n\in\mathbb{Z}_{\ge 1}$, its divisors not exceeding $\sqrt{n}$ are called \textit{small}. The phrase \textit{small divisors}, as used here, is not to be confused with classical small divisor problems of mathematical physics. In a recent paper, Iannucci \cite{Ian} showed a nice and surprising result that charaterizes all natural numbers whose small divisors are in arithmetic progression (AP). In particular, Iannucci defined
\begin{align}\label{kk}
    S_n\ :=\ \{d\ :\ d\ |\ n, d\le \sqrt{n}\}.
\end{align}
and analyzed the divisor-counting function to argue about the prime factorization of $n$ when $|S_n|\le 6$. Then he showed that if $S_n$ is in AP, $|S_n|$ cannot be greater than $6$ and finished the proof. For previous work on divisors in AP, see \cite{BFL, Var} and on small divisors, see \cite{BTK, Ian2}.

By Definition \eqref{kk}, $1$ is in $S_n$. This trivial divisor gives information about the AP and plays a crucial role in the argument of Iannucci. Our goal is to exclude $1$ and $\sqrt{n}$ from consideration to produce a more general theorem. For a natural number $n$, define
$$A_n\ :=\ \{d\,:\, d\ |\ n, 1<d<\sqrt{n}\}.$$
We shall determine all $n$ such that $A_n$ is in AP; that is, 
$$A_n \ =\ \{d, a+d, 2a+d, \ldots, (k-1)a+d\},$$
where $d$ is the first term, $a > 0$ is the common difference and $k$ is the number of terms in the AP. Then the $i$th term in $A_n$ is $d+(i-1)a$. Observe that if we write $n = p_1^{a_1}\cdots p_\ell^{a_\ell}$, where $p_1<p_2<\cdots< p_\ell$ are primes and $a_i\ge 1$, then $d = p_1$. For easy reading, we stick with these notation throughout the paper. By hypothesis, $A_n$ is in AP and $n\ge 2$. 
The following is our main theorem.
\begin{thm}\label{main0}
For all $n\in\mathbb{Z}_{\ge 1}$, we have $|A_n|\le 5$. In addition, if we let $p$, $q$, and $r$ denote distinct prime numbers, then one of the following is true:
\begin{itemize}
    \item [(i)] $n = p$ or $n = p^2$ for some $p$, hence $A_n = \emptyset$. 
    \item [(ii)] $n = pq$ for some $p<q$, hence $A_n = \{p\}$.
    \item [(iii)] $n = p^3$ or $n = p^4$ for some $p$, hence $A_n  = \{p\}$.
    \item [(iv)] $n = p^5$ for some $p$, hence $A_n = \{p, p^2\}$.
    \item [(v)] $n = pq^2$, where $p<q$, hence $A_n = \{p,q\}$.
    \item [(vi)] $n = p^2q$, where $p^2<q$, hence $A_n = \{p, p^2\}$.
    \item [(vii)] $n = p^2q$, where $p < q< p^2$, hence $A_n = \{p, q\}$.
    \item [(viii)] $n = p^6$ for some  $p$, hence $A_n = \{p, p^2\}$.
    \item [(ix)] $n = 36$, hence $A_{n} = \{2, 3, 4\}$.
    \item [(x)] $n = pqr$, where $p<q<r$ such that $2q = p+r$, hence $A_n = \{p, q, r\}$. (We have infinitely many triple of primes $(p, q, r)$ such that $2q = p+r$ because there are arbitrarily long arithmetic progressions of primes \cite{GT}.)
    \item [(xi)] $n = 24$, hence $A_n = \{2, 3, 4\}$.
    \item [(xii)] $n = 60$, hence $A_n = \{2,3,4,5,6\}$.
\end{itemize}
\end{thm}
Since $A_n\subset S_n$ for any $n\in \mathbb{Z}_{\ge 1}$, it is straightforward to deduce \cite[Theorem 4]{Ian} from our Theorem \ref{main0}. As usual, we have the divisor-counting function
$$\tau(n) \ :=\ \sum_{d|n}1.$$ Since $\tau(n)$ is multiplicative, for the $m$ distinct primes $p_1<p_2<\cdots<p_m$, and natural numbers $a_1, a_2, \ldots, a_m$,
\begin{align}\label{multau}\tau(p_1^{a_1}p_2^{a_2}\cdots p_m^{a_m})\ =\ (a_1+1)(a_2+1)\cdots (a_m+1).\end{align}
If $n = bc$ and $b\le c$, then $b\le \sqrt{n}\le c$; hence,
\begin{align*}
    \tau(n)\ =\ \begin{cases} 2|S_n|-1&\mbox{ if }n \mbox{ is a square};\\
    2|S_n| &\mbox{ if } n\mbox{ is not a square}.
    \end{cases}
\end{align*}
Since 
\begin{align*}
    |S_n|\ =\ \begin{cases} |A_n|+2&\mbox{ if }n \mbox{ is a square};\\
    |A_n|+1 &\mbox{ if } n\mbox{ is not a square},
    \end{cases}
\end{align*}
we have
\begin{align}
    \label{k1}
    \tau(n)\ =\ \begin{cases} 2|A_n|+3&\mbox{ if }n \mbox{ is a square};\\
    2|A_n|+2 &\mbox{ if } n\mbox{ is not a square}.
    \end{cases}
\end{align}
In Section \ref{kkk}, we prove that for all $n\in\mathbb{Z}_{\ge 1}$, $|A_n|\le 5$. In Section \ref{kkk} and the Appendix, we analyze $|A_n| = k$ for each $0\le k\le 5$ to prove Theorem \ref{main0}. Since the case analysis when $k = 5$ is lengthy and not interesting, we move it to the Appendix. 

\section{Proof of the first part of Theorem \ref{main0}}\label{kkk}
In this section, we prove that $|A_n|\le 5$. Note that this bound is sharp since $A_{60} = \{2, 3, 4, 5, 6\}$. We state an immediate consequence of \cite[Lemma 2]{Ian} that will be used in due course. 
\begin{lem}\label{0218L1}
If $(d, a) = (2, 1)$ or $(3,2)$, then $|A_n|\le 5$.
\end{lem}

\begin{lem}\label{k2}
If $a = 1$ or $a = 2$, then $|A_n| = k\le 5$. 
\end{lem}
\begin{proof}
For a contradiction, we assume that $|A_n|\ge 6$. If $n$ has exactly one prime divisor, then write $n = p^u$ for some prime $p$ and $u\ge 2k+1$. So, $A_n = \{p, p^2, p^3, \ldots, p^{k}\}$. However, $p, p^2, p^3$ are not in AP for all $p$. Hence, $n$ must have at least two distinct prime divisors, called $p_1$ and $p_2$, where $p_1 < p_2$. 
\begin{enumerate}
\item First, we consider $a = 1$. The two smallest numbers in $A_n$ are $p_1$ and $p_1 + 1$. Since $p_1< p_2$, it follows that $p_1+1\le p_2$. There are two possibilities:  
\begin{itemize}
    \item[(i)] If $p_1 + 1$ is not a prime, then $p_1 + 1 \ge p_1^2$ because $(p_1+1)| n$. For all primes $p$, we have $p^2>p+1$, so this case cannot happen. 
    \item[(ii)] If $p_1 + 1$ is a prime, then $p_1 + 1 = p_2$. So, $p_1 = 2$ and $p_2 = 3$. It follows that $A_n = \{2,3,4,\ldots, m\}$ for $m\ge 7$. However, by Lemma \ref{0218L1}, we have $m\le 6$, a contradiction.
\end{itemize}

\item Suppose that $a=2$. The three smallest numbers in $A_n$ are $p_1, p_1+2$ and $p_1+4$. Suppose that $n$ has three distinct prime factors. It follows that $p_1+2\le p_2$; otherwise, $(p_1, p_2) = (2, 3)$, which, by Lemma \ref{0218L1}, implies that $|A_n|\le 5$, a contradiction. Since $p_1 + 2\le p_2$, we have $p_1 + 4 \le p_2 + 2\le p_3$. There are two possibilities for $p_1+2$:
\begin{itemize}
\item[(i)] If $p_1 +2$ is not a prime, then $p_1 + 2 \ge p_1^2$, which implies that $p_1=2$. Then the three smallest numbers in $A_n$ are $2, 4, 6$. However, $6\ |\ n$ implies that $3\ |\ n$, so $3\in A_n$, a contradiction.  
    \item [(ii)] If $p_1 + 2$ is a prime, then $p_1 + 2 = p_2$. We consider two subcases.
    \begin{itemize}
        \item[(a)]If $p_1 + 4 = p_3$, then at least one of $p_1, p_2, p_3$ is divisible by $3$. It must be that $p_1 = 3$. So, $A_n = \{3, 5, 7,\ldots, 2m-1\}$ for some $m\ge 7$. By Lemma \ref{0218L1}, we have a contradiction. 
        \item[(b)] If $p_1+4\ <\ p_3$, 
    Write $p_1 + 4 = p_1^{u_1} p_2^{u_2}$ for some $u_1, u_2\ge 0$. Then 
    \begin{align*}p_1^{u_1}p_2^{u_2} - 2 &\ =\ p_1 + 2\ =\ p_2\mbox{, and so}\\
    p_1^{u_1}p_2^{u_2} &\ =\ p_2 + 2.\end{align*}
    Clearly, either $u_1$ or $u_2 = 0$. If $u_1 = 0$, $p_2^{u_2} = p_2 + 2$ has no odd prime solutions $p_2$ for all $u_2\ge 0$. If $u_2 = 0$, $p_1^{u_1} = p_2 + 2 = p_1 + 4$ has no prime solutions $p_1$ for all $u_1\ge 0$.  
    \end{itemize}
\end{itemize}
So, $n$ has exactly two distinct prime factors $p_1$ and $p_2$. We write $p_1+4 = p_1^{u_1}p_2^{u_2}$ for some $u_1, u_2\ge 0$. As above, we have a contradiction. 
\end{enumerate}
Therefore, $k = |A_n| \le 5$, as desired. 
\end{proof}


\begin{lem}\label{k3} If $k := |A_n|\ge 6$, then $a$ is even and $p_1\ge 3$, where $a$ is the common difference and $p_1$ is the least term in $A_n$.   
\end{lem}

\begin{proof}
Due to the proof of Lemma \ref{k2}, $n$ has at least one prime factor other than $p_1$. Let $p_2$ denote the prime factor. We consider two cases corresponding to whether $p_1\ge 3$ or $p_1 = 2$. 
\begin{enumerate}
\item Assume $p_1\ge 3$. Since $p_1 + a\le p_2$, we have either $p_1 + a = p_2$ (if $p_1 + a$ is a prime) or $p_1 + a = p_1^2$ (if $p_1+a$ is not a prime). In both cases, $a$ is even.

\item Assume $p_1 = 2$. Then $A_n = \{2, 2+a, 2+2a, \ldots, 2+(k-1)a\}$. We show a contradiction. The third term in $A_n$ is $2(a+1) = p_1(a+1)$. Hence, we have $a+1\ |\ n$ and $a+1 < a+2$, so $a+1 = 2$ and $a=1$. By Lemma \ref{k2}, $k\le 5$, a contradiction. 
\end{enumerate}
We have shown that if $|A_n| = k\ge 6$, $a$ must be even and $p_1 \ge 3$.
\end{proof}

\begin{lem}\label{k4}
If $a> 2$ and $k\ge 6$, then $k\ge a+3$. 
\end{lem}
\begin{proof}
Suppose otherwise. That is, $a>2$ and $k\ge 6$, but $k\le a+2$. We will show a contradiction after proving the two following statements (i) and (ii).
\begin{itemize}
\item [(i)] $p_1\ \nmid\ a$. 
\item [(ii)] $n$ has at least $k - \floor{(k-1)/3}$ distinct prime factors. 
\end{itemize}

Proof of (i): Suppose that $p_1\ |\ a$. Write $a = up_1$ for some $u\in\mathbb{Z}_{\ge 1}$. Then
$$A_n \ =\ \{p_1, (u+1)p_1, (2u+1)p_1, \ldots, ((k-1)u+1)p_1\}.$$
Hence, $u+1$ and $2u+1\in A_n$. Since $u+1 < (u+1)p_1$, we have $u+1 = p_1$. Furthermore, since $u+1 < 2u+1$, we have $2u+1\ge (u+1)p_1$, hence $p_1\le 1$, which contradicts the primeness of $p_1$. This contradiction tells us that the assumption $p_1 | a$ is false.  So, $p_1$ does not divide $a$.

Proof of (ii): Define $$S := \{p_1 + ja\ :\ 1\le j\le k-1\}\subset A_n,$$ $S_1 := \{s\in S\ :\ p_1\ |\ s\}$, and $S_2 := S\backslash S_1$. Due to claim (i), $p_1 | s$ for some $s\in S$ if and only if $s = p_1 + ja$ for some $1\le j\le k-1$ and $p_1$ divides $j$. So, $|S_1| = \floor{(k-1)/p_1} \le \floor{(k-1)/3}$ (due to Lemma \ref{k3}) and so, we have
$$|S_2| \ =\ |S| - |S_1| \ \ge\ (k-1) - \floor{(k-1)/3}.$$
We claim that $s$ is prime for all $s\in S_2$. Otherwise, if $s = ja+p_1\in S_2$ is composite, then there exist $1\le b,c < j$ such that $ja+p_1 = (ba+p_1)(ca+p_1)$. Note that $b, c\ge 1$ because $p_1\ \nmid \ s$ if $s\in S_2$. We have
\begin{align*}
    (k-1)a+p_1 \ &\ge\ ja+p_1\ =\ (ba+p_1)(ca+p_1)\\
    \ &\ge\ (a+p_1)^2 \ =\ (a+2p_1)a+p_1^2\ >\ (k-1)a+p_1,
\end{align*}
a contradiction. The last inequality is due to the assumption that $a+2\ge k$. Therefore, all numbers in $S_2$ are prime; together with $p_1$, we know that $n$ has at least $k-\floor{(k-1)/3}$ distinct prime factors. This completes the proof of statement (ii).

We now deduce a contradiction, which rejects our supposition that $k\le a+2$. Due to statement (ii), we have $\tau(n)\ge 2^{k-\floor{(k-1)/3}}$. By \eqref{k1}, we obtain $$k\ =\ |A_n|\ \ge\ \frac{\tau(n)-3}{2}\ \ge\ \frac{2^{k-\floor{(k-1)/3}}-3}{2},$$
which contradicts our assumption that $k\ge 6$. This contradiction shows $k\ge a+3$.
\end{proof}

\begin{lem}\label{k5}
If $a>2$, then $|A_n| = k\le 5$. 
\end{lem}
\begin{proof}
Assume, by contradiction, that $a>2$ and $k\ge 6$. By Lemma \ref{k3}, $a$ is even; hence, $a\ge 4$. By Bertrand's postulate, there exists a prime $q$ such that $a/2< q< a$. Let 
$$Q\ =\ \{p_1 + a, p_1 + 2a, \ldots, p_1+qa\}.$$
Since $q<a$, by Lemma \ref{k4}, $Q\subset A_n$. Since $q\ \nmid\ a$, it follows that no two elements of $Q$ are congruent modulo $q$. Hence, there exists $j$ such that $1\le j\le q$ such that $q\ |\ ja+p_1$. Therefore, $q\ |\ n$ and $q < \sqrt{n}$. So, $q\in A_n$. The only possibility is $q = p_1$ because $q<a$. 

Observe that $p_1 + qa = q(a+1)$, so $a+1\in A_n$. Write $$A_n = \{q, q + a, q + 2a, \ldots, q + (k-1)a\}$$
to see that $a+1\notin A_n$. We have a contradiction. Therefore, $k\le 5$. 
\end{proof}

\begin{proof}[Proof of the first part of Theorem \ref{main0}]
The claim that $|A_n|\le 5$ follows immediately from Lemma \ref{k2} and Lemma \ref{k5}.
\end{proof}
\section{Proofs of the second part of Theorem \ref{main0}}
We consider natural numbers $n$ such that $|A_n|$ does not exceed $5$ and $A_n$ is in AP. We prove by case analysis with each case corresponding to each $0\le k\le 4$. For each value of $k$, we may itemize further if necessary. For $k = 5$, see the Appendix. 
\begin{proof}[Proof of the second part of Theorem \ref{main0}] We consider 5 cases. 

\begin{enumerate}
\item Case 1: $k=0$. By \eqref{k1}, $\tau(n) = 2$ or $3$. By \eqref{multau}, $n = p$ or $n = p^2$ for some prime $p$. In both cases, $A_n = \emptyset$, which is vacuously in AP. This corresponds to item (i) of Theorem \ref{main0}.

\item Case 2: $k=1$. By \eqref{k1}, $\tau(n) = 4$ or $5$. By \eqref{multau}, $n = pq$, $p^3$ or $p^4$ for some primes $p<q$. In all cases, $A_n = \{p\}$, which is vacuously in AP. This corresponds to items (ii) and (iii). 

\item Case 3: $k=2$. By \eqref{k1}, $\tau(n) = 6$ or $7$. By \eqref{multau}, $n = p^5$, $pq^2$, $p^2q$ or $p^6$ for some primes $p<q$. 
\begin{itemize}
    \item[(a)] If $n = p^5$, then $A_n = \{p, p^2\}$, which is in AP. This is item (iv).
    \item[(b)] If $n = pq^2$, then $A_n = \{p,q\}$, which is in AP. This is item (v).
    \item[(c)] If $n = p^2q$, then we have two subcases.
    \begin{itemize}
        \item[(i)] If $p^2<q$, then $A_n = \{p, p^2\}$. This is item (vi).
        \item[(ii)] If $p^2>q$, then $A_n = \{p, q\}$. This is item (vii).
    \end{itemize}
    \item[(d)] If $n = p^6$, then $A_n = \{p, p^2\}$, which is in AP. This is item (viii). 
\end{itemize}

\item Case 4: $k = 3$. By \eqref{k1}, $\tau(n) = 8$ or $9$. By \eqref{multau}, $n = pqr$, $pq^3$, $p^3q$, $p^7$, $p^2q^2$ or $p^8$ for some primes $p<q<r$. 
\begin{itemize}
    \item[(a)] If $n = p^7$ or $p^8$, because $p, p^2, p^3$ are not in AP, we eliminate this case. 
    \item[(b)] If $n = p^2q^2$, then $A_n = \{p, p^2, q\}$. 
    \begin{itemize}
    \item[(i)] If $q>p^2$, then in order that $A_n$ is in AP, $q = 2p^2- p = p(2p-1)$, which contradicts that $q$ is a prime. We eliminate this case.
    \item[(ii)] If $p< q < p^2$, then $2q = p^2+p = p(p+1)$. If $p\ge 3$, then $q = p\frac{p+1}{2}$ is not a prime because $\frac{p+1}{2}\ge 2$. Hence, $p = 2$, $q = 3$, and $n = 36$. This is item (ix).
    \end{itemize}
    \item[(c)] If $n = pqr$, either $r>pq$ or $r<pq$.
    \begin{itemize}
        \item[(i)] If $r> pq$, then $A_n = \{p, q, pq\}$. Since $A_n$ is in AP, we have $p = 2q/(q+1) < 2$, a contradiction. 
        \item[(ii)] If $r< pq$, then $A_n = \{p, q, r\}$. Then $2q = p+r$. This is item (x). 
    \end{itemize}
    \item[(d)] If $n = pq^3$, then $A_n = \{p, q, pq\}$. So, $2q = p+pq$ and so, $p = 2q/(q+1)<2$, a contradiction.
    \item[(e)] If $n = p^3q$, either $q>p^3$ or $q<p^3$. 
    \begin{itemize}
        \item[(i)] If $q> p^3$, then $A_n = \{p, p^2, p^3\}$, which is not in AP. 
        \item[(ii)] If $q < p^3$, then $A_n = \{p, p^2, q\}$. If $q> p^2$, then $A_n$ is in AP if and only if $2p^2 = p+q$. So, $q = p(2p-1)$, which contradicts that $q$ is a prime. If $q< p^2$, then $A_n$ is in AP if and only if $2q = p+p^2 = p(p+1)$. Equivalently, $q = p\frac{p+1}{2}$. In order that $q$ is prime, we must have $p = 2$, $q = 3$, and $n = 24$. This is item (xi).
    \end{itemize}
\end{itemize}

\item Case 5: $k = 4$. By \eqref{k1}, $\tau(n) = 10$ or $11$. By \eqref{multau}, $n = p^9$, $pq^4$, $p^4q$ or $p^{10}$ for some primes $p<q$. 
\begin{itemize}
    \item[(a)] If $n = p^9$ or $p^{10}$, then $A_n$ is not in AP because $p, p^2, p^3$ are not in AP. We eliminate this case. 
    \item[(b)] If $n = pq^4$, then $A_n = \{p, q, pq, q^2\}$. It follows that $q+pq = p+q^2$. So, $p = q$, which contradicts our assumption that $p<q$. 
    \item[(c)] If $n = p^4q$, then we consider three subcases.
    \begin{itemize}
        \item[(i)] If $p< q< p^2$, then $A_n = \{p, q, p^2, pq\}$. In order that $A_n$ is in AP, $q + pq = p^2 + p$ and so, $p = q$, a contradiction. 
        \item[(ii)] If $p^2<q< p^4$, then $A_n = \{p, p^2, p^3, q\}$. In order that $A_n$ is in AP, the three numbers $p$, $p^2$, and $p^3$ cannot be the three smallest elements of $A_n$ because $\{p, p^2 p^3\}$ are not in AP. Hence, $p^2<q<p^3$. We have $p^3 + p = p^2 + q$. So, $p | q$, a contradiction. 
        \item[(iii)] If $p^4< q$, then $A_n= \{p, p^2, p^3, p^4\}$, which is not in AP. 
    \end{itemize}
\end{itemize}
\end{enumerate}
We have analyzed all cases up to $k = 4$. For $k = 5$, see the Appendix. 
\end{proof}

\section{Acknowledgement} The author would like to thank the anonymous referees for helpful suggestions that improve the exposition of this paper. 

\appendix
\section{When $|A_n| = 5$}
By \eqref{k1}, $\tau(n) = 12$ or $13$. By \eqref{multau}, $n = pq^5$, $p^5q$, $p^{2}q^3$, $p^3q^2$, $p^2qr$, $pq^2r$ or $pqr^2$ for some primes $p<q<r$. We can eliminate the cases $n = p^{11}$ and $n = p^{12}$ because $\{p, p^2, p^3\}$ are not in AP.
\begin{enumerate}

\item If $n = pq^5$, then $A_n = \{p, q, pq, q^2, pq^2\}$ with $p < q< pq < q^2 < pq^2$. So, $p+pq = 2q$ and so $p = 2q/(q+1)<2$, a contradiction.

\item If $n = p^5q$, we consider the following subcases.
\begin{itemize}
    \item[(a)] If $q< p^3$, then $A_n = \{p, p^2, p^3, q, pq\}$. 
    \begin{itemize}
    \item[(i)] If $q<p^2$, then $p < q< p^2< pq < p^3$. So, $2q = p+p^2 = p(p+1)$, which gives $(p,q) = (2,3)$. However, $A_{96}$ is not in AP.
    \item[(ii)] If $p^2<q< p^3$, then $p<p^2<q<p^3<pq$. So, $2q = p^2 + p^3 = p^2(p+1)$, which implies that $p|q$, which contradicts that $q$ is a prime greater than $p$.  
    \end{itemize}
    \item[(b)] If $p^3<q$, then the three smallest numbers in $A_n$ are $p, p^2, p^3$, which are not in AP. 
\end{itemize}

\item If $n = p^2q^3$, either $p^2>q$ or $p^2<q$. 
    \begin{itemize}
        \item[(a)] If $p<q<p^2$, then $A_n = \{p, q, p^2,  pq, q^2\}$ with $p < q< p^2 < pq < q^2$. So, $2pq = p^2 + q^2$, which implies that $p(2q-p) = q^2$ and so, $p|q$, a contradiction. 
        \item[(b)] If $p^2<q$, then $A_n = \{p, p^2, q, pq, p^2q\}$ with $p < p^2 < q< pq < p^2q$. So, $2q = p^2 + pq$; thus, $p = 2q/(p+q) < 2$, a contradiction. 
    \end{itemize}
    
\item If $n = p^3q^2$, either $p^3> q^2$ or $p^3<q^2$. 
    \begin{itemize}
        \item[(a)] If $p^2< q^2< p^3$, then $A_n = \{p, q, p^2, pq, q^2\}$ with $p < q< p^2 < pq < q^2$. So, $2pq = p^2 + q^2$, which implies that $p| q$, a contradiction. 
        \item[(b)] If $p^3< q^2$, then $A_n = \{p, p^2, p^3, q, pq\}$. 
        \begin{itemize}
        \item[(i)] If $q< p^2$, we have $p < q< p^2< pq < p^3$. So, $2pq = p^3 + p^2$; equivalently, $q = p\frac{p+1}{2}$, which gives $(p,q) = (2, 3)$. However, $A_{72} = \{2, 3, 4, 6\}$, which is not in AP.
        \item[(ii)] If $p^2< q< p^3$, we have $p<p^2<q<p^3<pq$. So, $2q = p^2+p^3 = p^2(p+1)$, which contradicts the primeness of $q$. 
        \item[(iii)] If $p^3 < q$, we have $p<p^2<p^3<q<pq$. Since $p, p^2, p^3$ are not in AP, we eliminate this case. 
        \end{itemize}
    \end{itemize}


\item If $n = p^2qr$, either $p^2q>r$ or $p^2q<r$. 
    \begin{itemize}
        \item[(a)] If $p^2q>r$, then $A_n = \{p, p^2, q, pq, r\}$. If $q>p^2$, the three smallest numbers in $A_n$ are $p < p^2 < q$. In order that $A_n$ is in AP, $2p^2 = p+q$ and so, $p | q$, a contradiction. So, $p<q<p^2$. We consider two further subcases $q < r< p^2$ and $p^2<r$. 
        \begin{itemize}
            \item[(i)] If $p<q<r<p^2<pq$, then $q+pq = p^2 + r$ and $2q = p+r$. Substituting $r = 2q - p$ into the former, we have $pq = p^2 + q - p$, which implies that $p\ |\ q$, a contradiction.  
            \item[(ii)] If $p^2<r$, then the three smallest numbers are $p < q< p^2$. So, $2q = p(p+1)$, which gives $(p,q) = (2,3)$. 
            \begin{itemize}
            \item If $r<pq$, then $r = 5$ and $n = 60$. This is item (xii). 
            \item If $r>pq$, then $p<q<p^2<pq<r$. So, $2pq = p^2 + r$ and so, $r = 8$, which contradicts that $r$ is prime.
            \end{itemize}
        \end{itemize}
        \item[(b)] If $p^2q<r$, then $A_n = \{p, p^2, q, pq, p^2q\}$. Similar as above, we have $p<q<p^2$. So, 
        $p<q<p^2<pq<p^2q$. Since the three smallest numbers are $p< q< p^2$, we have $2q = p + p^2$, which gives $(p,q) = (2,3)$. However, $2,3,4,6,12$ are not in AP. 
    \end{itemize}


\item If $n = pq^2r$, we consider two cases $r>pq$ and $r<pq$.
    \begin{itemize}
        \item[(a)] If $r>pq$, then the three smallest numbers in $A_n$ are $p<q<pq$. So, $2q = p+pq$, which gives $p = 2q/(q+1)<2$, a contradiction. 
        \item[(b)] If $r<pq$, then the four smallest numbers in $A_n$ are $p<q<r<pq$. So, $2r = pq+q = q(p+1)$, which implies that either $r|q$ or $r|(p+1)$. However, both cases contradict that $p < q < r$.
    \end{itemize}


\item If $n = pqr^2$, then $A_n = \{p, q, pq, r, pr\}$. Either $p<q<pq<r<pr$ or $p<q<r<pq<pr$. None of these cases (which we have seen above) gives us an $A_n$ in AP.
\end{enumerate}

\end{document}